\documentclass[12pt]{article}
\usepackage{amsfonts}
\usepackage{amssymb}
\usepackage{mathrsfs}
\usepackage{amsmath}
\usepackage{amsthm}
\usepackage{amstext}
\usepackage{amsopn}
\usepackage{graphicx}
\usepackage{color}
\newtheorem{theorem}{Theorem}[section]
\newtheorem{lemma}[theorem]{Lemma}
\newtheorem{proposition}[theorem]{Proposition}

\theoremstyle{definition}

\newtheorem{example}[theorem]{Example}

\theoremstyle{remark}
\newtheorem{remark}[theorem]{Remark}

\numberwithin{equation}{section}

\newcommand{\ba}{\begin{array}}
\newcommand{\ea}{\end{array}}
\newcommand{\f}{\frac}

\newcommand{\Om}{\Omega}

\newcommand{\la}{\lambda}

\newcommand{\ds}{\displaystyle}

\begin{document}
\date{}
\title{ \bf\large{Hopf bifurcation for a delayed diffusive logistic population model in the advective heterogeneous environment}\footnote{This research is supported by the National Natural Science Foundation of China (No 11771109)}}
\author{Shanshan Chen\footnote{Email: chenss@hit.edu.cn},\ \  Junjie Wei\footnote{Corresponding Author, Email: weijj@hit.edu.cn},\ \ Xue Zhang
 \\
{\small  Department of Mathematics,  Harbin Institute of Technology,\hfill{\ }}\\
 {\small Weihai, Shandong, 264209, P.R.China.\hfill{\ }}}
\maketitle

\begin{abstract}
{In this paper, we investigate a delayed reaction-diffusion-advection equation, which models the population
dynamics in the advective heterogeneous environment. The existence of the nonconstant positive steady state
and associated Hopf bifurcation are obtained.
A weighted inner product associated with the advection rate is introduced to compute the normal forms, which is the main difference between Hopf bifurcation for delayed reaction-diffusion-advection model and that for delayed reaction-diffusion model.
Moreover, we find that the spatial scale and advection can affect Hopf bifurcation in the heterogenous environment.
}

\end{abstract}

\section {Introduction}

In recent decades, there are extensive works on the population
dynamics in the advective environments. For example, the population may have a tendency towards better quality habitat, and Belgacem and Cosner \cite{Belgacem1995} proposed the following model
\begin{equation}\label{taxis}
\begin{cases}
\ds\frac{\partial u}{\partial t} =\nabla\cdot[d\nabla u-au\nabla m]+
u\left[m(x)-u\right],&x\in\Om,\ t>0,\\
u(x,t)=0,&  x\in\partial\Om,\ t>0,\\
\end{cases}
\end{equation}
where $a$ measures the tendency of the population to move up or down along the gradient of $m(x)$.
We refers to \cite{CantrellB,Cantrell2006,ChenHL2008,Cosner2003,Lamk2011} and the references therein for results on this type of advection.
Moreover, in streams and rivers, the unidirectional water flow always exists and
can influence the population
dynamics of the river species \cite{JinY2014,Lutscher2006,Lutscher2007,JinY2012}.
Lou and Zhou \cite{LouY2015} considered the following single species model,
\begin{equation}\label{flowR}
\begin{cases}
\ds\frac{\partial u}{\partial t} =du_{xx}-\alpha u_x+
u\left(r-u\right),&0<x<L,\ t>0,\\
du_x(0,t)-\alpha u(0,t)=0,& t>0,\\
du_x(L,t)-\alpha u(L,t)=-b\alpha u(L,t), &t>0,
\end{cases}
\end{equation}
where $u(x,t)$ denotes the population density at location $x$ and time $t$, $d>0$ is the diffusion rate, $r>0$ represents the intrinsic growth rate, $x=0,L$ are the upstream end and downstream ends respectively, $\alpha$ accounts for the advection rate
caused by the unidirectional water flow, and $b$ measures the lose of the species at the downstream end. Eq. \eqref{flowR} can also model the population dynamics of a species in a water column, where $x$ runs from the top ($x=0$) to the bottom ($x=L$).
Therefore, $\alpha$ may be positive or negative depending on whether the density of the species is heavier or lighter
than the water \cite{ZhouP2016}. If $b\to \infty$, the hostile boundary condition at the downstream is obtained, and Speirs and Gurney
showed \cite{Speirs2001} that
the species can persist only when the speed of the flow is slow and the stream is long. If $b=1$, the boundary condition is referred to as the free-flow boundary condition or the Danckwerts boundary condition, see \cite{Vasilyeva2011} for detailed analysis on persistence. For more general case, Lou and Zhou \cite{LouY2015} gave the necessary and sufficient condition for the persistence of the species with respect to $b$. We also refer to \cite{LouY2014,ZhouLX2016,LouY2015,ZhouP2016,ZhouP2017} and the references therein for results on two competing species with this type of advection.

For reaction-diffusion equations without advection term, it is well-known that time delay can make the constant steady states or nonconstant steady states unstable, and spatial homogeneous or nonhomogeneous periodic solutions can occur through Hopf bifurcation, see \cite{faria2000,Faria2001,gourley2004nonlocality,Hadeler,Hu,Lee,Sen} and the references therein. Especially, Busenberg
and Huang \cite{busenberg1996stability} first studied the Hopf bifurcation near the nonconstant positive steady state, and they found that, for the following single population model,
\begin{equation}\label{pp}
\begin{cases}
\ds\frac{\partial u(x,t)}{\partial t} =d\Delta u(x,t)+r
u(x,t)\left(1-u(x,t-\tau)\right),&x\in\Om,\ t>0,\\
u(x,t)=0,&  x\in\partial\Om,\ t>0,\\
\end{cases}
\end{equation}
time delay $\tau$ can induce Hopf bifurcation, see also \cite{HuR,su2009hopf,Su2011,yan2010stability,yan2012} for some more general population models.
we also refer to \cite{Chen2012,Chen2016,Guo2015,Guo2017,Guo2016} for the Hopf bifurcation of models with the nonlocal delay effect and homogenous Dirichlet boundary conditions.
A natural question is that whether delay can induce instability for reaction-diffusion-advection models. For model \eqref{taxis}, considering the delay effect, Chen et al. \cite{Chen2017} studied the following model
\begin{equation}\label{taxisde}
\begin{cases}
\ds\frac{\partial u}{\partial t} =\nabla\cdot[d\nabla u-au\nabla m]+
u\left(m(x)-u(x,t-\tau)\right),&x\in\Om,\ t>0,\\
u(x,t)=0,&  x\in\partial\Om,\ t>0,\\
\end{cases}
\end{equation}
and showed that Hopf bifurcation is more likely to occur when the advection rate increases.

In this paper, we mainly concern whether delay can induce Hopf bifurcation for model \eqref{flowR}, and for simplicity we only consider the case of $b=0$. Actually, we investigate the
following model for a single species in the advective heterogeneous enviroment
\begin{equation}\label{delay1}
\begin{cases}
u_t =du_{xx}-\alpha u_x+
u\left(m(x)-\int_0^LK(x,y)u(y,t-\tau)dy\right),&0<x<L,\ t>0,\\
du_x(0,t)-\alpha u(0,t)=0,\;du_x(L,t)-\alpha u(L,t)=0,& t>0,\\
\end{cases}
\end{equation}
where parameters $d$, $\alpha$ and $L$ have the same meanings as that in model \eqref{flowR}, delay $\tau$ represents the maturation time, and intrinsic growth rate $m(x)$ is spatially dependent and show the effect of the heterogenous environment.
Here $K(x,y)$ accounts for the nonlocality of the species. We remark that this kind of nonlocal effect is not induced by the time delay, and it represents
the nonlocal interspecific
competition of the species for resources. The individuals at different
locations may compete for common resource or communicate either visually or by
chemical means, see \cite{britton1990spatial,Furter} for the detailed biological explaination.
Throughout the paper, unless otherwise specified, we assume that $m(x)$ satisfies:
\begin{enumerate}
\item[$\mathbf{(A_1)}$] $m(x)\in C^2[0,L]$, and $m(x)\ge(\not\equiv)0$,
\end{enumerate}
and the following assumption is imposed on the kernel function $K(x,y)$:
\begin{enumerate}
\item[$(\mathbf{A}_2)$] either
$$
K(x,y)=\delta(x-y),
$$
or
$$
~K(x,y)\in L^\infty((0,L)\times (0,L)),
$$
where $L_+:=\{(x,y)\in(0,L)\times (0,L):K(x,y)>0\}$ has positive Lebesgue measure.
\end{enumerate}
For example, the following kernel function
\begin{equation}\label{K}
K(x,y)=\begin{cases}
0,\;\;y>x\\
1,\;\;0<y\le x\\
\end{cases}
\end{equation}
satisfies assumption $(\mathbf{A}_2)$, and was used to model the nonlocal competition of the phytoplankton for light \cite{DuHsu2010,Huisman}. Moreover, if $K(x,y)=\delta(x-y)$, then
$$
\int_0^LK(x,y)u(y,t-\tau)dy=u(x,t-\tau),
$$
and there is no nonlocal effect.

For the case that advection $\alpha=0$ and $K(x,y)=\delta(x-y)$, Shi et al. \cite{shishis} showed that delay can induced Hopf bifurcation for model \eqref{delay1}. Our main results also extend the results of \cite{busenberg1996stability,shishis}, and show that
Hopf bifurcation can also occur at the nonconstant positive
steady state when $\alpha\ne0$. Moreover, we will show that if $m(x)$ is spatially dependent, then the spatial scale and advection can affect Hopf bifurcation. For example, Hopf bifurcation can be more likely to occur when the advection rate increases or decreases for different types of $m(x)$. This phenomenon is different from that in
model \eqref{taxisde}, where Hopf bifurcation is more likely to occur when the advection rate increases. We point out that, since the boundary condition is different, the method and arguments in \cite{Chen2012} should be modified to investigate this model.

Letting $\tilde u=e^{(-\alpha /d)x}u$, $\tilde t=d t$, denoting $\tilde r=1/d$, $\tilde \alpha=a/d$, $\tilde\tau=d\tau$,  and dropping the tilde sign, model \eqref{delay1} can be transformed as the following equivalent model:
\begin{equation}\label{delay}
\begin{cases}
u_t=e^{-\alpha x}\left(e^{\alpha x} u_x\right)_x+
r u\left(m(x)-\int_0^L K(x,y)e^{\alpha y}u(y,t-\tau)dy\right),\;\;x\in(0,L),\ t>0,\\
u_x(0,t)=u_x(L,t)=0,\;\;t>0.\\
\end{cases}
\end{equation}
The initial value of model \eqref{delay} is
\begin{equation}\label{initial1}u(x,s)=\eta(x,s)\ge0,\;\;\;\;
x\in(0,L),\ t\in[-\tau,0],\end{equation}where $\eta\in
C:=C([-\tau,0],Y)$ and $Y=L^2(0,L)$. Note that $e^{-\alpha x}\f{\partial }{\partial x}\left(e^{\alpha x} \f{\partial }{\partial x}\right)$ generates an analytic
semigroup $T(t)$ on $Y$ with the domain
\begin{equation}
\mathscr{D}\left(e^{-\alpha x}\f{\partial }{\partial x}\left(e^{\alpha x} \f{\partial }{\partial x}\right)\right)=
\{\psi\in H^2(0,L):\psi_x(0)=\psi_x(L)=0\}.
\end{equation}
Define $F:C\rightarrow Y$ by
\begin{equation}
F(\Psi)(x)=r
\Psi(0)\left(m(x)-\int_0^LK(x,y)e^{\alpha y}\Psi(-\tau)(y)dy\right).
\end{equation}
An easy calculation implies that $F$ is locally Lipschitz continuous. Therefore, it follows from
\cite{wu1996theory} that, for each $\Psi\in C$, there exists a
maximum $t_{\Psi}>0$ such that model \eqref{delay}
has a unique solution $u_{\Psi}(t)$ existing on $[-\tau,t_{\Psi})$.
The following eigenvalue problem is crucial for our further investigation
\begin{equation}\label{eigen-1}
\begin{cases}
-e^{-\alpha x}\left(e^{\alpha x}\phi_x\right)_x=\la m(x) \phi(x),\;\;x\in(0,L),\\
\phi_x(0)=\phi_x(L)=0.
\end{cases}
\end{equation}
Denote by $\la_1$ the principal eigenvalue of problem \eqref{eigen-1}, and let $\phi$ be the corresponding eigenfunction with respect to $\la_1$ such
that $\phi(x)>0$. It follows from \cite{LouY2015} that
\begin{equation}\label{la1}
\la_1=\inf_{0\ne\psi \in W^{1,2}}\ds\f{\int_0^L e^{\alpha x}\psi_x^2dx}{\int_0^Le^{\alpha x}\psi^2dx}=0,
\end{equation}
$\phi$ is a constant, and we choose $\phi=1$ for simplicity.

For simplicity of the notations, as in \cite{Chen2012},
we also denote the spaces $$X=\{\psi\in H^2(0,L):\psi_x(0)=\psi_x(L)=0\},$$
$Y=L^2(0,L)$, $C=C([-\tau,0],Y)$, and $\mathcal{C}=C([-1,0],Y)$ throughout the paper.
Let the complexification of a linear space $Z$ be $Z_\mathbb{C}:= Z\oplus
iZ=\{x_1+ix_2|~x_1,x_2\in Z\}$, and define the domain
of a linear operator $T$ by $\mathscr{D}(T)$, the kernel of $T$ by $\mathscr{N}(T)$, and the range of $T$ by $\mathscr{R}(T)$.
Moreover, for Hilbert space $Y_{\mathbb C}$,
the standard inner product is $\langle u,v \rangle=\ds\int_{0}^L \overline
u(x) {v}(x) dx$.
The rest of the paper is organized as follows. In Section 2, we
show that a nonconstant positive steady state bifurcates from the trivial equilibrium. The Hopf bifurcation near this nonconstant positive steady state is also investigated.
In Section 3, we obtain
the direction of the Hopf bifurcation and the stability of the bifurcating periodic orbits. In Section 4, the effect of spatial heterogeneity are obtained, and the spatial scale and advection can affect Hopf bifurcation in the heterogenous environment.
Moreover, some numerical simulations are given to illustrate our theoretical results. Especially, Eq. \eqref{delay1} can model the population dynamics for a species in a water column with nonlocal competition for light.
We numerically show that when advection rate $\alpha=0$, the density of the species concentrates on the top of the water column. However when $\alpha$ is large, the density of the species concentrates on the bottom of the water column.

\section {Stability and Hopf bifurcation}
\subsection{Positive steady states and eigenvalue problem}
Firstly, we show the existence of positive steady states of Eq.
\eqref{delay}, which satisfy
\begin{equation} \label{steady}
\begin{cases}
\left(e^{\alpha x}u_x\right)_x+r e^{\alpha x}u\left(m(x)-\int_0^LK(x,y)e^{\alpha y}u(y)dy\right)=0, &
x\in(0,L),\\
u_x(0)=u_x(L)=0.
\end{cases}
\end{equation}
Denote
\begin{equation}\label{LL}
P_0:=\f{\partial }{\partial x}\left(e^{\alpha x} \f{\partial }{\partial x}\right).
\end{equation}
Then
\begin{equation*}
      X  =  \mathscr{N}\left(P_0\right)\oplus X_1,\;\;
      Y  = \mathscr{N}\left(P_0\right)\oplus Y_1,
\end{equation*}
 where
\begin{equation}\label{L}
\begin{split}
\mathscr{N}\left(P_0\right)=&\text{span}\{\phi\}=\text{span}\{1\},\;\;
X_1=\left\{y\in X:\int_{0}^L y(x)dx=0\right\},\\
Y_1=&\mathscr{R}\left(P_0\right)=\left\{y\in Y:\int_{0}^L
y(x)dx=0\right\}.
\end{split}
\end{equation}
By the arguments similar to Theorem A.2. of \cite{Cantrell1996}, we obtain the existence of positive steady states in the following.
\begin{theorem}\label{Tsteady}
There exist $r_1>0$ and a continuously differentiable mapping $r\mapsto u_r$ from $[0,r_1]$ to $X$ such that $u_r$ is a positive solution of Eq. \eqref{steady} for $r\in(0,r_1]$, and $u_0=c_0$, where
\begin{equation}\label{ovec}
c_0=\ds\f{\int_0^L m(x) e^{\alpha x}dx}{\int_0^L\int_0^L K(x,y) e^{\alpha x+\alpha y}dxdy}>0.
\end{equation}
\end{theorem}
\begin{proof}
It follows from assumption $(\mathbf{A})$ that $c_0>0$.
Define $H: \mathbb R\times X_1\times \mathbb R\to Y$ by
$$H(c,w,r)=P_0w+re^{\alpha x}(c+w)\left(m(x)-\int_0^L K(x,y)e^{\alpha y}(c+w(y))dy\right).$$
Letting
\begin{equation}\label{spe}
u= c+w,\;\; c\in\mathbb R,\;w\in X_1,
 \end{equation}
and substituting it
into Eq. \eqref{steady}, we see that $(u,r)$ solves Eq. \eqref{steady}, where $u\in X$, $r>0$,
if and only if $H(c,w,r)=0$
is solvable for some value of $c\in \mathbb R$, $w\in X_1$ and $r>0$.
Note that $H(c,0,0)=0$ for any $c\in \mathbb R$. An easy calculation implies that
\begin{equation*}
\begin{split}
D_{(w,r)}H(c,w,r)[v,\sigma]=&P_0v+rm(x)e^{\alpha x} v-re^{\alpha x}(c+w)\int_0^LK(x,y)e^{\alpha y}v(y)dy\\
-&re^{\alpha x}v\int_0^L K(x,y)e^{\alpha y}(c+w(y))dy\\
+&\sigma e^{\alpha x}(c+w)\left(m(x)-\int_0^LK(x,y)e^{\alpha y}(c+w(y))dy\right).
\end{split}
\end{equation*}
Here
$D_{(w,r)}H(c,w,r)$
is the Fr\'echet derivative of $H(c,w,r)$ with respect to $(w,r)$.
Then,
$$D_{(w,r)}H(c,0,0)[v,\sigma]=P_0 v+\sigma ce^{\alpha x}\left(m(x)-c\int_0^L K(x,y) e^{\alpha y}dy\right).$$
Since
$$-c_0e^{\alpha x}\left(m(x)-c_0\int_0^L K(x,y) e^{\alpha y}dy\right)\in Y_1=\mathscr{R}\left(P_0\right),$$
there exists a unique $v^*\in X_1$ such that
$$P_0v^*=-c_0e^{\alpha x}\left(m(x)-c_0\int_0^L K(x,y) e^{\alpha y}dy\right),$$
and consequently,
$$\mathscr{N}( D_{(w,r)}H(c_0,0,0))=\{(sv^*,s):s\in\mathbb R\}.$$
A direct computation yields
$$D_cD_{(w,r)}H(c_0,0,0)[v^*,1]=m(x)e^{\alpha x}-2c_0 e^{\alpha x}\int_0^L K(x,y)e^{\alpha y}dy,$$
where $D_cD_{(w,r)}H(c_0,0,0)$ is the Fr\'echet derivative of $D_{(w,r)}H(c,w,r)$ with respect to $c$ at $(c_0,0,0)$.
We claim that
\begin{equation}\label{cl1}D_cD_{(w,r)}H(c_0,0,0)[v^*,1]\not\in \mathscr{R}\left(D_{(w,r)}H(c_0,0,0)\right).
\end{equation}
 Suppose it is not true. Then, there exists $(\tilde v,\tilde \sigma)$ such that
\begin{equation}
\begin{split}
D_{(w,r)}H(c_0,0,0)[\tilde v,\tilde \sigma]=&P_0 \tilde v+\tilde \sigma c_0e^{\alpha x}\left(m(x)-c_0\int_0^L K(x,y) e^{\alpha y}dy\right)\\
=&m(x)e^{\alpha x}-2c_0 e^{\alpha x}\int_0^L K(x,y)e^{\alpha y}dy,
\end{split}
\end{equation}
which implies that
$$m(x)e^{\alpha x}-2c_0 e^{\alpha x}\int_0^L K(x,y)e^{\alpha y}dy\in \mathscr{R}(P_0).$$
This contradicts with the fact that
$$\int_0^Lm(x)e^{\alpha x}dx-2c_0\int_0^L\int_0^LK(x,y)e^{\alpha x+\alpha y}dxdy=-\int_0^L m(x)e^{\alpha x}dx\ne0.$$
Therefore, Eq. \eqref{cl1} holds, and it follows from the Crandall-Rabinowitz bifurcation theorem
\cite{Crandall} that the solutions of $H(c,w,r)=0$ near $(c_0,0,0)$ consist precisely by the curves
$\{(c,0,0):c\in \mathbb {R}\}$ and $$\{(c(s),w(s),r(s)): s\in(-\delta,\delta)\},$$ where $(c(s),w(s),r(s))$ are continuously differentiable, $c(0)=c_0$, $w(0)=0$, $r(0)=0$, $w'(0)=v^*$, and $r'(0)=1$.
Since $r'(0)=1>0$, $r(s)$ has a inverse function $s(r)$ for small $s$. Noticing that $c_0>0$, we see that there exists $r_1>0$ such that Eq. \eqref{steady} has a positive solution $u_r=c(s(r))+w(s(r))$ for $r\in(0,r_1]$. Moreover, $$u_0=c(s(0))+w(s(0))=c(0)+w(0)=c_0.$$ This completes the proof.
\end{proof}

%\begin{proposition}\label{ulamsta}%Let $u_\la$ be the positive spatially nonhomogeneous
%equilibrium obtained in Theorem \ref{Tsteady}. Then, for any $\la\in(0,r_1]$,
% $u_\la$ is locally asymptotically stable when $\tau=0$.
%\end{proposition}
\begin{remark}
It follows from the imbedding theorem that $u_r\in C^{1+\delta}([0,L])$ for some $\delta\in(0,1)$, and
$\lim_{r\to 0} u_r=c_0$ in $C^{1+\delta}([0,L])$.
\end{remark}
Then, we obtain the eigenvalue problem associated with $u_r$. The Linearized equation of \eqref{delay} at $u_r$ takes the following form
\begin{equation}
\label{linear}\begin{cases}\ds\frac{\partial v}{\partial t} =
e^{-\alpha x}P_0v+
r
\left(m(x)-\int_0^L K(x,y) e^{\alpha y}u_r(y)dy\right)v\\~~~~-r u_r\int_0^L K(x,y)e^{\alpha y}
v(y,t-\tau)dy,& x\in(0,L),\; t>0,\\
v_x(x,t)=0,&x=0,L,\; t>0.
\end{cases}
\end{equation}
Denote
\begin{equation}\label{Ar}
\Tilde K(r):= m(x)-\int_0^L K(x,y) e^{\alpha y}u_{r}(y)dy.
\end{equation}
From \cite{wu1996theory}, we see that the solution semigroup of
Eq. \eqref{linear} has the infinitesimal generator $A_\tau(r)$ defined by
\begin{equation}\label{Ataula}A_\tau(r) \Psi=\dot\Psi\end{equation}
with the domain
\begin{equation*}\begin{split}
 &\mathscr{D}(A_\tau(r)) = \big\{\Psi\in C_\mathbb{C}
\cap C^1_\mathbb{C}:\ \Psi(0)\in X_{\mathbb{C}},\dot\Psi(0)=e^{-\alpha x}P_0\Psi(0)+r\Tilde K(r)\Psi(0)\\
&~~~~~~~~~~~~~-r u_r\int_0^L K(x,y)e^{\alpha y}
\Psi(-\tau)(y)dy  \big\},
\end{split}\end{equation*}
where $
C^1_\mathbb{C}=C^1([-\tau,0],Y_\mathbb{C})$, $P_0$ and $\Tilde K(r)$ are defined as in Eqs. \eqref{LL} and \eqref{Ar} respectively. Moreover, $\mu\in\mathbb{C}$ is an eigenvalue of $A_\tau(r)$, if and only if there exists $\psi(\ne0)\in
X_{\mathbb{C}}$ such that $\Delta(r,\mu,\tau)\psi=0$,
where
\begin{equation}\label{triangle}
\begin{split}
&\Delta(r,\mu,\tau)\psi
:=e^{-\alpha x}P_0\psi+r \Tilde K(r)\psi
-r u_r\int_0^L K(x,y)e^{\alpha y}
\psi(y)dy e^{-\mu\tau}-\mu\psi.
\end{split}
\end{equation}
Then $A_\tau(r)$ has a purely imaginary eigenvalue $\mu=i\nu\
(\nu>0)$ for some $\tau\ge0$, if and only if
\begin{equation}\label{eigen}
\begin{split}
&P_0\psi+r e^{\alpha x}\Tilde K(r)\psi
-r u_re^{\alpha x}\int_0^L K(x,y)e^{\alpha y}
\psi(y)dy e^{-i\theta}-i\nu e^{\alpha x}\psi =0
\end{split}
\end{equation}
is solvable for some value of $\nu>0$, $\theta\in[0,2\pi)$, and $\psi(\ne 0)\in X_{\mathbb{C}}$.
The estimates for solutions of Eq. \eqref{triangle} can be derived as follows.
\begin{lemma}\label{nu}
Assume that $(\mu_r,\tau_r,\psi_r)$ solves $\Delta(r,\mu,\tau)\psi=0$ with ${\mathcal Re}\mu_r,\tau_{r}\ge0$ and $0\ne \psi_r
\in X_{\mathbb{C}}$.  Then
$\left|\ds
\frac{\mu_r}{r}\right|$ is bounded for
$r\in(0,r_1]$.\end{lemma}
\begin{proof}
Noticing that $u_{r}$ is the principal eigenfunction of $P_0+re^{\alpha x}\Tilde K(r)$ with principal eigenvalue $0$, we have $\langle \psi, P_0\psi+re^{\alpha x}\Tilde K(r)\psi\rangle\le0$ for any $\psi\in X_{\mathbb C}$.
Substituting $(\mu_r,\tau_r,\psi_r)$ into $\Delta(r,\mu,\tau)\psi=0$, multiplying it by $e^{\alpha x}\overline\psi_r$, and integrating the result over $(0,L)$, we have
\begin{equation}\label{inqem2}
\begin{split}
&\langle \psi_{r}, P_0\psi_{r}+re^{\alpha x}\Tilde K(r)\psi_{r}\rangle\\
=&r \int_0^L\int_0^L K(x,y)e^{\alpha x+\alpha y}u_r(x)\overline{\psi_{r}}(x)\psi_{r}(y)dxdye^{-\mu_r\tau_r}+\mu_{r}\int_0^L e^{\alpha x}|\psi_{r}|^2dx.
\end{split}
\end{equation}
Since ${\mathcal Re}\mu_r,\tau_{r}\ge0$, we see that
\begin{equation*}
\begin{split}
0\le{\mathcal Re}(\mu_{r}/r)\le&\ds\f{1}{\int_0^L e^{\alpha x} |\psi_{r}|^2dx}{\mathcal Re}\left[-\int_0^L\int_0^L K(x,y)e^{\alpha x+\alpha y}u_r(x)\overline{\psi_{r}}(x)\psi_{r}(y)dxdye^{-\mu_r\tau_r}\right]\\
   \le&e^{2\alpha L}L\|u_r\|_\infty\|K(x,y)\|_\infty,
\end{split}
\end{equation*}
and
\begin{equation*}
\begin{split}
|\mathcal{I}m
(\mu_{r}/r)|=&\ds\f{1}{\int_0^L e^{\alpha x}|\psi_{r}|^2dx}\left|\mathcal{I}m\left[ \int_0^L\int_0^L K(x,y)e^{\alpha x+\alpha y}u_r(x)\overline{\psi_{r}}(x)\psi_{r}(y)dxdye^{-\mu_r\tau_r}\right]\right|\\
   \le&e^{2\alpha L}L\|u_r\|_\infty\|K(x,y)\|_\infty.
\end{split}
\end{equation*}
It follows from the continuity of
$r\mapsto \|u_r\|_\infty$ that $\ds
\left|\frac{\mu_r}{r}\right|$ is bounded for $r\in(0,r_1]$.
\end{proof}
The following result is similar to Lemma 2.3 of \cite{busenberg1996stability} and
we omit the proof here.
\begin{lemma}\label{lem21}
Assume that $z\in (X_1)_{\mathbb C}$. Then $|
\langle P_0z,z\rangle|\geq
\la_2\|z\|^2_{Y_{\mathbb{C}}}$, where $\la_2$ is the second
eigenvalue of operator $-P_0$.
\end{lemma}

%\end{proof}
For $r\in(0,r_1]$, ignoring a scalar factor, $\psi$ in Eq. \eqref{eigen} can be represented as
\begin{equation}
\label{eigen2}
\begin{split}
&\psi= \beta c_0+r z,\;\;\; z\in (X_1)_{\mathbb{C}},\; \; \beta\geq0, \\
 &\|\psi\|^2_{Y_{\mathbb{C}}}=\beta^2 c_0^2L
 +r^2\|z\|^2_{Y_{\mathbb{C}}}= c_0^2L,
 \end{split}
 \end{equation}
 where $c_0$ is defined as in Eq. \eqref{ovec}.
Then, substituting the first Equation of \eqref{eigen2} and $\nu=rh$
into Eq. \eqref{eigen}, we obtain that $(\nu,\theta,\psi)$ solves Eq. \eqref{eigen}, where $\nu>0$, $\theta\in[0,2\pi)$ and $\psi\in X_{\mathbb{C}}(\|\psi\|^2_{Y_{\mathbb{C}}}= c_0^2L)$,
if and only if the following system:
\begin{equation}\label{g1}
\begin{cases}g_1(z,\beta,h,\theta,r):=P_0z+e^{\alpha x}\Tilde K(r)(\beta c_0+rz)-ihe^{\alpha x}(\beta c_0+rz)
\\~~~~~~~~~~~~~~~~~~~-e^{\alpha x}u_r\int_0^L K(x,y)e^{\alpha y}(\beta c_0+rz(y))dy e^{-i\theta}=0\\
 g_2(z,\beta,r):=(\beta^2-1)c_0^2L+r^2\|z\|^2_{Y_{\mathbb{C}}}=0
\end{cases}
\end{equation}
has a solution $(z,\beta,h,\theta)$, where $z\in (X_1)_{\mathbb{C}}$, $\beta\ge0$, $h>0$ and $\theta\in[0,2\pi)$.
Define
$G:(X_1)_{\mathbb C}\times \mathbb{R}^4\to
Y_{\mathbb C}\times \mathbb{R}$ by $G=(g_1,g_2)$. Note that $u_0=c_0$, and we first show that
$G(z,\beta,h,\theta,r)=0$ is uniquely solvable for $r=0$.
\begin{lemma}\label{l25}
The following equation \begin{equation}\label{3.5G}
\begin{cases}
G(z,\beta,h,\theta,0)=0\\
z\in (X_1)_{\mathbb{C}},\;h\ge0,\;\beta\ge0,\; \theta\in[0,2\pi]\\
\end{cases}
\end{equation} has a unique solution $(z_{0},\beta_{0},h_{0},\theta_{0})$, where
\begin{equation}\label{lastar}
    \beta_{0}=1,\;\;\theta_{0}=\pi/2,\;\;h_{0}=\ds\f{\int_0^L m(x)e^{\alpha x}dx}{\int_0^L e^{\alpha x}dx},
\end{equation}
and $z_{0}\in(X_1)_{\mathbb C}$ is the unique solution of
\begin{equation}\label{lastari}
P_0 z=-c_0e^{\alpha x}\left(m(x)-c_0\int_0^L K(x,y)e^{\alpha y}dy\right)-ic_0^2 e^{\alpha x}\int_0^L K(x,y)e^{\alpha y}dy+ih_0c_0 e^{\alpha x}.
\end{equation}
\end{lemma}
\begin{proof}
Obviously, $g_2(z,\beta,0)=0$ if and only if $\beta=\beta_{0}=1$. Then, substituting $\beta=\beta_0$ into
$g_1(z,\beta,h,\theta,0)=0$, we have
\begin{equation}\label{subequi}
P_0z=-c_0e^{\alpha x}\left(m(x)-c_0\int_0^L K(x,y)e^{\alpha y}dy\right)
+c_0^2 e^{\alpha x}\int_0^L K(x,y)e^{\alpha y}dy e^{-i\theta}+ihc_0e^{\alpha x}.
\end{equation}
It follows from Eq. \eqref{ovec} that
$$c_0\int_0^L\int_0^L K(x,y) e^{\alpha x+\alpha y }dxdy=\int_0^L m(x)e^{\alpha x}dx.$$
Then Eq. \eqref{subequi} has a solution $(z,h,\theta)$, where $z\in(X_1)_{\mathbb C}$, $h\ge0$, $\theta\in[0,2\pi]$, if and only if
\begin{equation}\label{cost}
\begin{cases}
c_0\int_0^L\int_0^L K(x,y) e^{\alpha x+\alpha y }dxdy\sin\theta=h\int_0^L e^{\alpha x}dx
\\
\int_0^L K(x,y) e^{\alpha x+\alpha y}dxdy\cos\theta=0\\
\end{cases}
\end{equation}
has a solution $(\theta,h)$ with $h\ge0$ and $\theta\in[0,2\pi]$,
which yields
\begin{equation}\label{hthe}
\theta=\theta_{0}=\pi/2,\;\;h=h_{0}=\ds\f{\int_0^L m(x)e^{\alpha x}dx}{\int_0^L e^{\alpha x}dx}.
\end{equation}
Substituting $h=h_0$ and $\theta=\theta_0$ into Eq. \eqref{subequi}, we see that the right side of Eq. \eqref{subequi} belongs to $\mathscr{R}\left(P_0\right)$, which implies that $z=z_0$.
\end{proof}
Then, we show that
$G(z,\beta,h,\theta,r)=0$ is also uniquely solvable for small $r$.
\begin{theorem}\label{cha}
There exist $r_2>0$ and a continuously differentiable mapping
$r\mapsto(z_r,\beta_r,h_r,\theta_r)$ from
$[0,r_2]$ to $(X_1)_{\mathbb C}\times \mathbb{R}^3$ such that
$(z_r,\beta_r,h_r,\theta_r)$ is the unique solution of the following equation
\begin{equation}\label{3.6G} \begin{cases}
G(z,\beta,h,\theta,r)=0,\\
z\in (X_1)_{\mathbb{C}},\;h>0,\;\beta\ge0, \;\theta\in[0,2\pi),\\
\end{cases}
\end{equation}
for $r\in[0,r_2]$.
\end{theorem}
\begin{proof}
Denote the Fr\'echet derivative of $G$ with respect to
$(z,\beta,h,\theta)$ at $(z_{0},\beta_{0},h_{0},\theta_{0},0)$ by $T=(T_1,T_2):(X_1)_{\mathbb C}\times \mathbb{R}^3\mapsto
Y_{\mathbb C}\times \mathbb{R}$.
Then, a direct calculation leads to
\begin{equation*}
\begin{split}
T_1(\chi,\kappa,\epsilon,\vartheta)=&
P_0\chi+\kappa c_0 e^{\alpha x}\left[m(x)-c_0\int_0^L K(x,y)e^{\alpha y}dy\right]+i\kappa c_0^2 e^{\alpha x}\int_0^LK(x,y)e^{\alpha y}dy\\
-&ic_0\kappa e^{\alpha x}+\vartheta c_0^2e^{\alpha x}\int_0^LK(x,y)e^{\alpha y}dy-i\epsilon c_0e^{\alpha x},\\
 T_2(\kappa)=&2\kappa c_0^2L.
\end{split}
\end{equation*}
Obviously, $T$ is
a bijection from $(X_1)_{\mathbb C}\times \mathbb{R}^3$ to $Y_{\mathbb
C}\times \mathbb{R}$. It follows from the implicit function theorem that
there exist $r_2>0$ and a continuously differentiable mapping
$r\mapsto(z_r,\beta_r,h_r,\theta_r)$ from
$[0,r_2]$ to $X_{\mathbb C}\times \mathbb{R}^3$ such that
$G(z_r,\beta_r,h_r,\theta_r,r)=0$. Now, we show the uniqueness, and only need to prove that if $z^r\in (X_1)_{\mathbb{C}}$, $\beta^{r}\ge0$, $h^{r}>0$, $\theta^{r}\in [0,2\pi)$ satisfy $G(z^r,\beta^r,h^r,\theta^r,r)=0$,
then
$(z^r,\beta^r,h^r,\theta^r)\rightarrow(z_{0},1,h_0,\pi/2)$
as $r\rightarrow 0$ in $X_{\mathbb{C}}\times
\mathbb{R}^3.$ From Lemma \ref{nu} and Eq. \eqref{g1}, we obtain that $\{h^r\}, \{\beta^r\}$ and $\{\theta^r\}$ are
bounded for $r\in[0,r_1]$. Multiplying the first equation of \eqref{g1} by $\overline {z^r}$, and integrating the result over $(0,L)$, we obtain that there exist positive constants $M_1$ and $M_2$ such that
$\la_2\|z^r\|^2_{Y_{\mathbb{C}}}\leq |
\langle z^r,P_0z^r\rangle|\le
M_1\|z^r\|_{Y_{\mathbb{C}}}+
M_2r\|z^r\|^2_{Y_{\mathbb{C}}}
$
for $r\in(0,r_2]$, where $\la_2$ is defined as in Lemma \ref{lem21}.
Then, for sufficiently small $r_2$, $\{z^r\}$ is bounded in
$Y_{\mathbb C}$ for $r\in[0,r_2]$. Note that $P_0:(X_1)_{\mathbb C}\to (Y_1)_{\mathbb C} $ has a bounded inverse $P_0^{-1}$. Then,
$\{z^r\}$ is also bounded in $(X_1)_{\mathbb C}$, and
$\{(z^r,\beta^r,h^r,\theta^r): r\in(0,r_2]\}$ is
precompact in $Y_{\mathbb C}\times\mathbb{R}^3.$ Therefore, there exists
a subsequence $\{(z^{r^n},\beta^{r^n},h^{r^n},\theta^{r^n})\}_{n=1}^\infty$
 such that
$$
(z^{r^n},\beta^{r^n},h^{r^n},\theta^{r^n})
 \to(z^{0},\beta^{0},h^{0},\theta^{0})\;\;\text{in}\;\;Y_{\mathbb C}\times \mathbb{R}^3,$$  and $r^n\rightarrow0$ as $n\rightarrow\infty$.
Taking the limit of the equation
$$P_0^{-1}g_1(z^{r^n},\beta^{r^n},h^{r^n},\theta^{r^n},r^n)=0$$
as $n\rightarrow\infty$, we see that
$$(z^{r^n},\beta^{r^n},h^{r^n},\theta^{r^n})
 \to(z^{0},\beta^{0},h^{0},\theta^{0})\;\;\text{in}\;\;X_{\mathbb C}\times \mathbb{R}^3,$$  as $n\rightarrow\infty$,
and $(z^{0},r^{0},h^{0},\theta^{0})$ is also a solution of Eq. \eqref{3.5G}, which leads to
$$(z^{0},r^{0},h^{0},\theta^{0})=(z_{0},\beta_{0},h_{0},\theta_{0}).$$ This completes the proof.
\end{proof}
Finally, from Theorem \ref{cha}, we derive the following result.
\begin{theorem}\label{c25}
For $r\in(0,r_2],$ $(\nu,\tau,\psi)$ solves
\begin{equation*}
\begin{cases}
\Delta(r,i\nu,\tau)\psi=0,\\
\nu>0,\;\tau\ge0,\;\psi (\ne 0) \in X_{\mathbb C},\\
\end{cases}
\end{equation*}
if and
only if
\begin{equation}\label{par}
\nu=\nu_r=rh_r,\;\psi= a\psi_r,\;
\tau=\tau_{n}=\frac{\theta_r+2n\pi}{\nu_r},\;\; n=0,1,2,\cdots,
\end{equation}
where $\psi_r=\beta_rc_0+rz_r$,
$a$ is a nonzero constant, and
$(z_r,\beta_r,h_r,\theta_r)$ is defined as in Theorem
\ref{cha}.
\end{theorem}
\subsection{Distribution of the eigenvalues and Hopf bifurcation}
In this subsection, we will show the distribution of the eigenvalues of $A_\tau(r)$ and the existence of the Hopf bifurcation for model \eqref{delay}. Throughout this subsection, unless otherwise specified, we always assume $r\in(0,r_2]$, and the value of $r_2$ may be chosen smaller than the one in Theorem \ref{cha}, since further perturbation
arguments are used. Firstly, we show the distribution of the eigenvalues of $A_\tau(r)$ for $\tau=0$.
\begin{theorem}\label{nodelay1}
For $r\in(0,r_2],$ all the eigenvalues of
$A_{\tau}(r)$ have negative real parts when $\tau=0$.
\end{theorem}
\begin{proof}
To the contrary, there exists a sequence
$\{r^n\}_{n=1}^\infty$ such that $\ds\lim_{n\to\infty}r^n=0$, and for $n\ge1$,
$r^n>0$, and
corresponding eigenvalue problem
\begin{equation}\label{nodelay}
\begin{cases}
P_0\psi+r^ne^{\alpha x}\Tilde K(r^n)\psi
-r^ne^{\alpha x} u_{r^n}\int_0^L K(x,y)e^{\alpha y}
\psi(y)dy=\mu e^{\alpha x}\psi,\;\;x\in(0,L)\\
\psi_x(0)=\psi_x(L)=0
\end{cases}
\end{equation}
has an eigenvalue $\mu_{r^n}$ with ${\mathcal Re}\mu_{r^n}\ge0$, where $P_0$ and $\Tilde K(r)$ are defined as in Eqs. \eqref{LL} and \eqref{Ar} respectively. Ignoring a scalar factor, we assume that the associated
eigenfunction $\psi_{r^n}$ with respect to $\mu_{r^n}$ satisfies
$\|\psi_{r^n}\|^2_{Y_{\mathbb{C}}}=c_0^2L$, and
$\psi_{r^n}$ can be represented as $\psi_{r^n}=\beta_{r^n}c_0+r^nz_{r^n}$,
where $ \beta_{r^n}\ge0$, $z_{r_n}\in (X_1)_{\mathbb C}$ and $c_0$ is defined as in Eq. \eqref{ovec}.
As in Section 2.1, $\mu_{r^n}$ can also be represented as $\mu_{r^n}=r^n h_{r^n}$, and
it follows from Lemma \ref{nu} that $\left|h_{r^n}\right|$ is bounded for $r\in[0,r_2]$.
Then, substituting $\psi=\psi_{r^n}=\beta_{r^n}c_0+r^nz_{r^n}$ and $\mu=r^n h_{r^n}$ into the first equation of Eq.
\eqref{nodelay}, we see that $(z_{r^n},\beta_{r^n},h_{r^n})$ satisfies the following system
\begin{equation}\label{hi}
\begin{split}
&H_1(z,\beta,h,r_n):
=P_0 z+e^{\alpha x}\Tilde K(r^n)(\beta c_0+r^nz)\\&-e^{\alpha x}u_{r^n}\int_0^L K(x,y)e^{\alpha y}[\beta c_0+r^nz(y)]dy-he^{\alpha x}(\beta c_0+r^nz)=0,\\
&H_2(z,\beta,r_n)=(\beta^2-1)c_0^2L+(r^n)^2\|z\|^2_{Y_{\mathbb{C}}}=0.
\end{split}
\end{equation}
Using the arguments similar to Theorem \ref{cha}, we see that $(z_{r^n},\beta_{r^n},h_{r^n})$ is bounded in $Y_{\mathbb C}\times \mathbb R\times\mathbb C$. Since the
operator $P_0:(X_1)_{\mathbb C}\mapsto (Y_1)_{\mathbb C}$
has a bounded inverse $P_0^{-1}$, by applying $P_0^{-1}$ on
$$H_1(z_{r^n},\beta_{r^n},h_{r^n},r^n)=0,$$ we find that
$\{z_{r^n}\}_{n=1}^\infty$ is also bounded in $(X_1)_{\mathbb C}$, and consequently
$\{(z_{r^n},\beta_{r^n},h_{r^n})\}_{n=1}^\infty$ is
precompact in $Y_{\mathbb C}\times\mathbb{R}\times\mathbb{C}$. Therefore, there is
a subsequence $\{(z_{r^{n_k}},\beta_{r^{n_k}},h_{r^{n_k}})\}_{k=1}^\infty$
convergent to $(z^{*},\beta^{*},h^*)$ as $k\to \infty$ in the norm of $Y_{\mathbb C}\times\mathbb{R}\times\mathbb{C}$, where
$\beta^*=1$, $z^*\in Y_{\mathbb{C}}$ and $h^{*}\in\mathbb{C}$ with ${\mathcal Re}h^*\ge0$.
Taking the limit of the equation
$$P_0^{-1}H_1(z_{r^{n_k}},\beta_{r^{n_k}},h_{r^{n_k}})=0$$
as $k\rightarrow\infty$, we see that $z^*\in (X_1)_{\mathbb{C}}$ and $(z^*, \beta^*,h^*)$ satisfies
\begin{equation*}
\begin{split}
&P_0z^*+c_0e^{\alpha x}\left(m(x)-c_0\int_0^L K(x,y)e^{\alpha y}dy\right)\\
&-c_0^2e^{\alpha x}\int_0^LK(x,y)e^{\alpha y}dy-h^*c_0e^{\alpha x}=0.
\end{split}
\end{equation*}
Therefore,$$-c_0\int_0^L\int_0^L K(x,y) e^{\alpha x+\alpha y }dxdy=h^*\int_0^L e^{\alpha x}dx,$$ which leads to $h^*<0$. This contradicts with ${\mathcal Re}h^*\ge0$.
\end{proof}
Then, we show the distribution of the eigenvalues of $A_\tau(r)$ for $\tau>0$.
As in \cite{Chen2012}, one need to study the adjoint operator $\tilde\Delta(r,i\nu,\tau)$ of $e^{\alpha x}\Delta(r,i\nu,\tau)$, which takes the following form:
\begin{equation}\label{adj}
\tilde\Delta(r,i\nu,\tau)\tilde \psi=P_0\tilde \psi+r e^{\alpha x}\Tilde K(r)\psi
-r e^{\alpha x}\int_0^L K(y,x)u_r(y)e^{\alpha y}
\tilde \psi(y)dy e^{i\nu\tau}+i\nu e^{\alpha x}\tilde \psi.
\end{equation}
It follows that \begin{equation}\label{dual}
\langle
\tilde\psi,e^{\alpha x}\Delta(r,i\nu,\tau)\psi\rangle=
\langle\tilde\Delta(r,i\nu,\tau)\tilde\psi,\psi\rangle,
\end{equation}
for any $\tilde \psi, \psi\in X_{\mathbb C}$, and
$$\sigma_p(e^{\alpha x}\Delta(r,i\nu,\tau))=\sigma_p(\tilde\Delta(r,i\nu,\tau)).$$
Now, we consider the corresponding adjoint equation
\begin{equation}\label{eigendual}
P_0\tilde\psi+re^{\alpha x}\Tilde K(r)\tilde \psi-re^{\alpha x}\int_{0}^LK(y,x)e^{\alpha y}u_r(y)\tilde\psi(y)dy
e^{i\tilde\theta}+i\tilde\nu e^{\alpha x}\tilde\psi=0,\;0\ne\tilde\psi\in
X_{\mathbb{C}}.
\end{equation}
Note that if Eq. \eqref{eigendual} is solvable for some value of
$\tilde\nu>0$, $\tilde\theta\in[0,2\pi)$ and $\tilde\psi(\ne0)\in
X_{\mathbb{C}}$, then
$$
\tilde\Delta(r,i\tilde\nu,\tilde\tau_n)\tilde\psi=0,\;\;
\text{where}\;\;
\tilde\tau_n=\frac{\tilde\theta+2n\pi}{\tilde\nu},\;\;
n=0,1,2,\cdots.
$$
Similarly, ignoring a scalar factor, $\tilde \psi$ in Eq. \eqref{eigendual} can also be represented as
\begin{equation}
\label{eigen2dual}
\begin{split}
&\tilde \psi= \tilde \beta c_0+r \tilde z,\;\; \tilde z\in (X_1)_{\mathbb{C}},\;\;\tilde \beta\geq0, \\
 &\|\tilde \psi\|^2_{Y_{\mathbb{C}}}=\tilde \beta^2 c_0^2L
 +r^2\|\tilde z\|^2_{Y_{\mathbb{C}}}= c_0^2L,
 \end{split}
 \end{equation}
 where $c_0$ is defined as in Eq. \eqref{ovec}.
Then, substituting the first equation of \eqref{eigen2dual} and $\tilde \nu=r\tilde h$
into Eq. \eqref{eigendual}, we obtain that $(\tilde \nu,\tilde \theta,\tilde \psi)$ solves Eq. \eqref{eigendual}, where $\tilde \nu>0$, $\tilde \theta\in[0,2\pi)$ and $\tilde \psi\in X_{\mathbb{C}}(\|\tilde\psi\|^2_{Y_{\mathbb{C}}}= c_0^2L)$,
if and only if the following system:
\begin{equation}\label{g1dual}
\begin{cases}\tilde g_1(\tilde z,\tilde \beta,\tilde h,\tilde\theta,r):=P_0\tilde z+e^{\alpha x}\Tilde K(r)(\tilde \beta c_0+r\tilde z)+ihe^{\alpha x}(\tilde \beta c_0+r\tilde z)
\\~~~~~~~~~~~~~~~~~~~-e^{\alpha x}\int_0^L K(y,x)e^{\alpha y}u_r(y)(\tilde \beta c_0+r\tilde z(y))dy e^{i\tilde\theta}=0\\
 g_2(\tilde z,\tilde \beta,r):=(\tilde \beta^2-1)c_0^2L+r^2\|\tilde z\|^2_{Y_{\mathbb{C}}}=0
\end{cases}
\end{equation}
has a solution $(\tilde z,\tilde \beta,\tilde h,\tilde \theta)$, where $\tilde z\in (X_1)_{\mathbb{C}}$, $\tilde \beta\ge 0$, $\tilde h>0$, and $\tilde \theta\in[0,2\pi)$.
Define
$\tilde G:(X_1)_{\mathbb C}\times \mathbb{R}^4\to
Y_{\mathbb C}\times \mathbb{R}$ by $\tilde G=(\tilde g_1,\tilde g_2)$. By the arguments similar to Lemma \ref{l25}, we obtain
that $G(\tilde z,\tilde \beta,\tilde h,\tilde \theta,0)=0$ is also uniquely solvable.
\begin{lemma}\label{l25dual}
The following equation \begin{equation}\label{3.5Gdual}
\begin{cases}
\tilde G(\tilde z,\tilde \beta,\tilde h,\tilde \theta,0)=0\\
\tilde z\in (X_1)_{\mathbb{C}},\;\tilde h\ge0,\;\tilde \beta\ge0,\; \tilde \theta\in[0,2\pi]\\
\end{cases}
\end{equation} has a unique solution $(\tilde z_{0},\tilde \beta_{0},\tilde h_{0},\tilde \theta_{0})$, where
\begin{equation}\label{lastardual}
    \tilde \beta_{0}=1,\;\;\tilde\theta_{0}=\pi/2,\;\;\tilde h_{0}=h_0,
\end{equation}
and $\tilde z_{0}\in(X_1)_{\mathbb C}$ is the unique solution of
\begin{equation}\label{lastaridual}
P_0 z=-c_0e^{\alpha x}\left[m(x)-c_0\int_0^L K(x,y)e^{\alpha y}dy\right]+ic_0^2 e^{\alpha x}\int_0^L K(y,x)e^{\alpha y}dy-ic_0 e^{\alpha x}.
\end{equation}

\end{lemma}
The following results can also be proved similarly as in Theorems \ref{cha} and \ref{c25}.
\begin{theorem}\label{chadual}
\begin{enumerate}
  \item [$(I)$] There exists a continuously differentiable mapping
$$r\mapsto(\tilde z_r,\tilde \beta_r,\tilde h_r,\tilde \theta_r)$$ from
$[0,r_2]$ to $(X_1)_{\mathbb C}\times \mathbb{R}^3$ such that
$(\tilde z_r,\tilde \beta_r,\tilde h_r,\tilde \theta_r)$ is the unique solution of the following equation
\begin{equation}\label{3.6Gdual} \begin{cases}
\tilde G(\tilde z,\tilde \beta,\tilde h,\tilde \theta,r)=0,\\
\tilde z\in (X_1)_{\mathbb{C}},\;\tilde h>0,\;\tilde \beta\ge0, \;\tilde\theta\in[0,2\pi),\\
\end{cases}
\end{equation}
for $r\in[0,r_2]$.
  \item [$(II)$] For $r\in[0,r_2]$, the eigenvalue problem
$$
\tilde\Delta(r,i\tilde\nu,\tilde\tau)\tilde\psi=0,\;\;\tilde\nu>0,\;\;\tilde\tau\ge0,\;\;
0\ne \tilde\psi \in X_{\mathbb C}
$$
has a solution $(\tilde\nu,\tilde \tau,\tilde \psi)$ if and only if
\begin{equation}\label{pardual}
\tilde\nu=\tilde\nu_{r}=r\tilde h_r,\;\;\tilde \psi=a \tilde\psi_r,\;\;
\tilde\tau=\tilde\tau_{n}=\frac{\tilde\theta_r+2n\pi}{\tilde\nu_r},\;\;
n=0,1,2,\cdots,\;\;
\end{equation}
where $a$ is a nonzero constant, $\tilde\psi_r=\tilde\beta_r\phi+r\tilde z_r$, and $\tilde z_r,\tilde
\beta_r,\tilde h_r,\tilde\theta_r$ are defined as in Part $(I)$.
\end{enumerate}

\end{theorem}
For later application, we give a remark on $(\tilde h_r,\tilde\theta_r,\tilde\nu_r)$.
\begin{remark}\label{redual}
By the arguments similar to Remark 2.8 of \cite{Chen2012}, we see that $h_{r}=\tilde{h}_{r}$,
$\theta_{r}=\tilde{\theta}_{r}$, $\nu_r=\tilde\nu_r$ and $\tau_n=\tilde\tau_n$. Therefore, in the following, we will always use $(h_{r},\theta_{r},
\nu_{r},\tau_n)$ instead of the ones with tilde. Moreover, we remark that the corresponding solution $\psi_\la$ of $\Delta(r,i\nu_{r},\tau_{n})\psi=0$ may be different from $\tilde\psi$.
\end{remark}
Now, we show that $i\nu_r$ is simple.
\begin{theorem}\label{thm34a}
Assume that
$r\in(0,r_2]$. Then $\mu=i\nu_r$ is a simple
eigenvalue of $A_{\tau_{n}}(r)$ for $n=0,1,2,\cdots$, where $i\nu_r$ and $\tau_n$ are defined as in Theorem \ref{c25}.
\end{theorem}
\begin{proof}
From Theorem \ref{c25}, we obtain that $\mathscr{N}[A_{\tau_{n}}
(r)-i\nu_r]=\text{Span}[e^{i\nu_r\theta}\psi_r]$, where $\theta\in[-\tau_n,0]$ and $\psi_r$ is defined as in
Theorem \ref{c25}.
If
$\phi_1\in\mathscr{N}[A_{\tau_{n}}
(r)-i\nu_r]^2$,
then
$$
[A_{\tau_{n}}
(r)-i\nu_r]\phi_1\in\mathscr{N}[A_{\tau_{n}}(r)-i\nu_r]=
\text{Span}[e^{i\nu_r\theta}\psi_r],
$$
which implies that there exists a constant $a$ such that
$$
[A_{\tau_{n}}
(r)-i\nu_r]\phi_1=ae^{i\nu_r\theta}\psi_r.
$$
It follows that
\begin{equation}
 \label{eq32}
\begin{split}
\dot{\phi_1}(\theta)&=i\nu_r\phi_1(\theta)+ae^{i\nu_r\theta}\psi_r,
\ \ \ \ \theta\in[-\tau_{n},0], \\
 \dot{\phi_1}(0)&=
 e^{-\alpha x}P_0\phi_1(0)+r \Tilde K(r)\phi_1(0)-r u_{r}\int_0^L K(x,y)e^{\alpha y}\phi_1(-\tau_n)(y)dy.
 \end{split}
 \end{equation}
The first equation of Eq. \eqref{eq32} yields
\begin{equation}
\label{eq33}
\begin{split}
\phi_1(\theta)&=\phi_1(0)e^{i\nu_r\theta}+a\theta
e^{i\nu_r\theta}\psi_r,\\
\dot{\phi_1}(0)&=i\nu_r\phi_1(0)+a\psi_r.
\end{split}
\end{equation}
Then, it follows from Eqs. \eqref{eq32} and \eqref{eq33} that
\begin{equation}\label{provesimple}
\begin{split}
&e^{\alpha x}\Delta(r,i\nu_r,\tau_{n})\phi_1(0)\\=&P_0\phi_1(0)-i\nu_r e^{\alpha x}\psi_1(0)+re^{\alpha x}\Tilde K(r)\phi_1(0)-r e^{-i\theta_r}u_r\int_0^L K(x,u) e^{\alpha y}\phi_1(0)(y)dy
\\=&ae^{\alpha x}\left(\psi_{r}-re^{-i\theta_{r}}\tau_{n}u_r\int_0^L K(x,y)e^{\alpha y}\psi_{r}(y)dy\right).
\end{split}\end{equation}
Multiplying the above equation by $\overline{\tilde\psi_r}(x)$ and integrating the result over $(0,L)$, we see from
Eq. \eqref{dual} and Remark \ref{redual} that
\begin{equation}\label{sn}
\begin{split}
0&=\left\langle\tilde\Delta(r,i\tilde\nu,\tilde\tau_{n})\tilde\psi_r,\phi_1(0)\right\rangle=\left\langle \tilde\Delta(r,i\nu,\tau_{n})\tilde\psi_r,\phi_1(0)\right\rangle
=\left\langle\tilde\psi_r,e^{\alpha x}\Delta(r,i\nu,\tau_{n})\phi_1(0)\right\rangle\\
&=a\left(\int_0^Le^{\alpha x}\overline{\tilde\psi_r}\psi_rdy-r\tau_{n}e^{-i\theta_r}\int_0^L\int_0^Lu_r(x)
K(x,y) e^{\alpha x+\alpha y}\overline{\tilde\psi_r}(x)\psi_{r}(y)
dxdy\right)\\&:=aS_{n}(r).
\end{split}\end{equation}
It follows from Theorems \ref{cha}, \ref{c25} and \ref{chadual} that $\theta_{r}\to \pi/2$, $r\tau_n\to (\ds\f{\pi}{2}+2n\pi)$, $\psi_r,\tilde \psi_r\to c_0$ in $X_{\mathbb{C}}$ as $r\to0$. Therefore,
\begin{equation}\label{eq31}
\lim_{r\to0}S_n(r)=c_0^2\left[1+i\left(\frac{\pi}{2}+2n\pi\right)\right]\int_{0}^Le^{\alpha x}dx\ne0,
\end{equation}
which yields
$a=0$. Therefore,
$$
\mathscr{N}[A_{\tau_{n}}(r)-i\nu_r]^j
=\mathscr{N}[A_{\tau_{n}}(r)-i\nu_r],\;\;j=
2,3,\cdots,\;\; n=0,1,2,\cdots,
$$
and $\mu=i\nu_r$ is a simple eigenvalue of
$A_{\tau_{n}}$ for $n=0,1,2,\cdots.$
\end{proof}
Noticing that $\mu=i\nu_{r}$ is a simple eigenvalue of $A_{\tau_{n}}$, from
the implicit function theorem, we see that there
are a neighborhood $O_{n}\times D_{n}\times
H_{n}\subset\mathbb{R}\times\mathbb{C}\times X_{\mathbb{C}}$ of
$(\tau_{n},i\nu_r,\psi_r)$ and a continuously
differential function $(\mu(\tau),\psi(\tau)):O_{n}\rightarrow D_{n}\times
H_{n}$ such that $
\mu(\tau_{n})=i\nu_r$, $\psi(\tau_{n})=\psi_r$, and for each $\tau\in O_{n}$, the only eigenvalue of
$A_\tau(r)$ in $D_{n}$ is $\mu(\tau),$ and
\begin{equation}\label{eq34}
\begin{split}
&e^{\alpha x}\Delta(r,\mu(\tau),\tau)\psi(\tau)=P_0\psi(\tau)+re^{\alpha x}\Tilde K(r)\psi(\tau)-\mu(\tau) e^{\alpha x} \psi(\tau) \\
&-ru_re^{\alpha x}\int_0^LK(x,y)e^{\alpha y}\psi(\tau)(y)dye^{-\mu(\tau)\tau}=0.
\end{split}
\end{equation}
A direct calculation can lead to the transversality
condition, and here we omit the proof.
\begin{theorem}\label{thm35}
For
$r\in(0,r_2]$,
$
\frac{d{\mathcal Re}[\mu(\tau_{n})]} {d\tau}>0$, $n=0,1,2,\cdots$.
\end{theorem}

Then, from Theorems \ref{c25}, \ref{nodelay1}, \ref{thm34a}
and \ref{thm35}, we obtain the distribution of eigenvalues of $A_\tau(r)$.
 \begin{theorem}\label{thm36}
For
$r\in(0,r_2]$, the infinitesimal generator
$A_\tau(r)$ has exactly $2(n+1)$ eigenvalues with positive
real parts when $\tau\in(\tau_{n},\tau_{n+1}],\
n=0,1,2,\cdots.$
\end{theorem}
Finally, we obtain the stability of the positive steady state $u_{r}$, and the existence of the associated Hopf bifurcation. We remark that the Hopf bifurcation theorem for general PFDEs was proved in \cite{wu1996theory}.
%\begin{theorem}\label{thm36}
%For
%$r\in(r_*,r^*]$, the infinitesimal generator
%$A_\tau(r)$ has exactly $2(n+1)$ eigenvalues with positive
%real parts when $\tau\in(\tau_{n},\tau_{n+1}],\
%n=0,1,2,\cdots.$
%\end{theorem}
\begin{theorem}\label{thm37}For $r\in(0,r_2]$,
the positive steady state $u_r$ obtained in Theorem \ref{Tsteady}
is locally asymptotically stable when $\tau\in[0,\tau_{0})$, and
unstable when $\tau\in(\tau_{0},\infty)$. Moreover, when $\tau=\tau_n$,
$(n=0,1,2,\cdots)$, system \eqref{delay} occurs Hopf bifurcation at the positive steady state $u_{r}$.
\end{theorem}

\section{The properties of the Hopf bifurcation}

In this section, we obtain the direction of the Hopf bifurcation of Eq. \eqref{delay}
and the stability of the bifurcating periodic solutions, the methods used are motivated by
\cite{faria2001normal,faria2002normal,faria2002smoothness,hassard1981theory}. Here, unless otherwise specified,
we also assume $r\in(0,r_2]$ throughout this section, and the value of $r_2$ may be chosen smaller than the one
in Section 2, since further perturbation arguments are also used.
Letting $U(t)=u(\cdot,t)-u_r$, $t=\tau\tilde t$,
$\tau=\tau_n+\gamma$, and dropping the tilde sign, system \eqref{delay} can be transformed as follows:
\begin{equation}\label{ab}
\ds\frac{dU(t)}{dt}=\tau_n e^{-\alpha x}P_0 U(t)+\tau_nP_1U_t+J(U_t,\gamma),
\end{equation}
where $U_t\in \mathcal{C}=C([-1,0],Y)$, $P_0$ is defined as in Eq. \eqref{LL}, and
\begin{equation*}
\begin{split}
&P_1U_t:=r\Tilde K(r)U(t)-ru_r\int_0^L K(x,y) e^{\alpha y} U(t-1)(y)dy,\\
 &J(U_t,\gamma):=\gamma  e^{-\alpha x}P_0U_t+\gamma
P_1U_t-(\gamma+\tau_n)r U(t)\int_0^L K(x,y)e^{\alpha y}U(t-1)(y)dy.\\
\end{split}
\end{equation*}
Then Eq. \eqref{ab} occurs Hopf bifurcation near the zero equilibrium  when $\gamma=0$. The
linearized equation of \eqref{ab} for $\gamma=0$ is
\begin{equation}\label{lab}
\ds\frac{dU(t)}{dt}=\tau_n e^{-\alpha x}P_0U(t)+\tau_nP_1U_t.
\end{equation}
Denote by
$\mathcal {A}_{\tau_n}$ the infinitesimal generator of the solution semigroup for Eq. \eqref{lab}.
From \cite{wu1996theory}, we have \begin{equation*}\begin{split}
\mathcal{A}_{\tau_n} \Psi=&\dot\Psi,\\
\mathscr{D}(\mathcal{A}_{\tau_n})=&\Big\{\Psi\in \mathcal
{C}_\mathbb{C}
\cap \mathcal {C}^1_\mathbb{C}:\ \Psi(0)\in X_{\mathbb{C}},\dot\Psi(0)=\tau_n e^{-\alpha x}P_0 \Psi(0)+\tau_nP_1 U_t\Big\},
\end{split}\end{equation*}
where $\mathcal{C}^1_\mathbb{C}=C^1([-1,0],Y_\mathbb{C})$,
and the abstract form of Eq. \eqref{ab} is
\begin{equation}\label{abab}
\ds\frac{dU_t}{dt}=\mathcal{A}_{\tau_n}U_t+X_0J(U_t,\gamma),
\end{equation}
where
\begin{equation*}X_0(\theta)=\begin{cases}0, \;\;\; & \theta\in[-1,0),\\
I, \;\;\; &\theta=0.\\
\end{cases}\end{equation*}
In order to compute the normal forms, we need to introduce a weighted inner product for $Y_{\mathbb{C}}$:
$$\langle u,v \rangle_1=\ds\int_{0}^L e^{\alpha x}\overline
u(x) {v}(x) dx\;\;\text{ for }\;\;u,v\in Y_{\mathbb{C}}.$$
Here the weight function is concerned with advection rate $\alpha$, $Y_{\mathbb{C}}$ is also a Hilbert space with this product, and
$$\langle v, v\rangle \le \langle v,v\rangle_1\le e^{\alpha L}\langle v, v\rangle.$$
Following the methods of \cite{faria2002normal,Su2011}, we introduce the formal
duality $\langle\langle\cdot,\cdot\rangle\rangle$ in $\mathcal{C}$  by
\begin{equation}\label{bil}
\langle\langle\tilde\Psi,\Psi\rangle\rangle=\langle
\tilde\Psi(0),\Psi(0)\rangle_1-r\tau_n\int_{-1}^0\left\langle\tilde\Psi(s+1),u_r\int_0^L
K(\cdot,y)e^{\alpha y}\Psi(s)(y)dy\right\rangle_1 ds,
\end{equation}
for $\Psi\in \mathcal{C}_{\mathbb{C}}$ and $\tilde\Psi\in
\mathcal{C}_{\mathbb{C}}^*:= C([0,1],Y_{\mathbb{C}})$.
As in \cite{Hale1971}, we can compute the
formal adjoint operator $\mathcal{A}^*_{\tau_n}$ of $\mathcal{A}_{\tau_n}$ with respect to the formal duality.
We remark that
$\mathcal{A}^*_{\tau_n}$ is referred to as the formal adjoint operator of $\mathcal{A}_{\tau_n}$, if
\begin{equation}\label{Atauadjont}
\langle\langle
\mathcal{A}^*_{\tau_n}\tilde\Psi,\Psi\rangle\rangle=\langle\langle
\tilde\Psi,\mathcal{A}_{\tau_n}\Psi\rangle\rangle
\end{equation}
for any $\Psi\in\mathscr{D}(\mathcal{A}_{\tau_n})$ and $\tilde\Psi\in
\mathscr{D}(\mathcal{A}^*_{\tau_n})$.
\begin{lemma}\label{dualoperator}
The
formal adjoint operator $\mathcal{A}^*_{\tau_n}$ of $\mathcal{A}_{\tau_n}$ is defined by
$$\mathcal{A}^*_{\tau_n}\tilde\Psi(s)=-\dot{\tilde\Psi}(s)$$ with the domain
\begin{equation*}\begin{split}
 \mathscr{D}(\mathcal{A}^*_{\tau_n}) = \Big\{ &\tilde\Psi\in \mathcal{C}^*_\mathbb{C}
\cap (\mathcal{C}^*_\mathbb{C})^1:\tilde\Psi(0)\in X_\mathbb{C},-\dot{\tilde\Psi}(0)=\tau_n e^{-\alpha x}P_0\tilde \Psi(0)\\&+r \tau_n\Tilde K(r)\tilde \Psi(0)-r\tau_n \int_0^L K(y,x) e^{\alpha y} u_{r}(y)\tilde\Psi(1)(y) dy\Big\},
\end{split}\end{equation*}
where $(\mathcal{C}^*_\mathbb{C})^1=C^1([0,1],Y_\mathbb{C})$.
\end{lemma}
\begin{proof}
For $\Psi\in\mathscr{D}(\mathcal{A}_{\tau_n})$ and $\tilde\Psi\in
\mathscr{D}(\mathcal{A}^*_{\tau_n})$,
\begin{equation*}
\begin{split}
&\langle\langle
\tilde\Psi,\mathcal{A}_{\tau_n}\Psi\rangle\rangle\\
=&\left\langle
\tilde\Psi(0),
(\mathcal{A}_{\tau_n}\Psi)(0)\right\rangle_1-r\tau_n\int_{-1}^0\left\langle\tilde\Psi(s+1),
u_r\int_0^L K(x,y)e^{\alpha y}\dot\Psi(s)(y)dy\right\rangle_1 ds\\
\end{split}
\end{equation*}
\begin{equation*}
\begin{split}
=&\left\langle\tilde\Psi(0),\tau_n e^{-\alpha x}P_0 \Psi(0)+\tau_nP_1\Psi\right\rangle_1\\
-&r\tau_n\left[\left\langle\tilde\Psi(s+1), u_r \int_0^L K(x,y) e^{\alpha y}\Psi(s)(y) dy\right\rangle_1\right]_{-1}^{0}\\
+&r\tau_n\int_{-1}^0\left\langle\dot{\tilde\Psi}(s+1),
u_r \int_0^L K(x,y)e^{\alpha y}\Psi(s)(y)dy\right\rangle_1 ds\\
=&\left\langle
(\mathcal{A}^*_{\tau_n}\tilde\Psi)(0),\Psi(0)\right\rangle_1
-r\tau_n\int_{-1}^0\left\langle-\dot{\tilde\Psi}(s+1),
u_r \int _0^L K(x,y)e^{\alpha y}\Psi(s)(y)dy\right\rangle_1 ds\\
=&\langle\langle
\mathcal{A}^*_{\tau_n}\tilde\Psi,\Psi\rangle\rangle.
\end{split}
\end{equation*}
This completes the proof.
\end{proof}
It follows from Theorem \ref{thm36} that $\mathcal {A}_{\tau_n}$ has
only one pair of simple purely imaginary eigenvalues $\pm i \nu_r\tau_n$, and the associated eigenfunction with respect to
$i\nu_r\tau_n$ (respectively, $-i\nu_r\tau_n$) is $\psi_r e^{ i \nu_r\tau_n\theta}$
(respectively, $\overline{\psi_r} e^{-i\nu_r\tau_n\theta}$) for $\theta\in[-1,0]$, where
$\psi_r$ is defined as in Theorem \ref{c25}.
Similarly, it follows from Theorem \ref{chadual}, Remark \ref{redual} and Lemma \ref{dualoperator} that the operator $\mathcal {A}^*_{\tau_n}$ also has only one pair of simple purely imaginary
eigenvalues $\pm i \nu_r\tau_n$, and the corresponding
eigenfunction with respect to $-i\nu_r\tau_n$ (respectively, $i\nu_r\tau_n$) is $\tilde\psi_r(x)
e^{i \nu_r\tau_n s}$ (respectively,
$\overline{\tilde\psi_r}(x) e^{i \nu_r\tau_ns}$) for $s\in[0,1]$,
where $\tilde\psi_r$ is defined in Theorem \ref{chadual}.
From \cite{wu1996theory}, we see that the center subspace of Eq. \eqref{ab} is
$P=\text{span}\{p(\theta),\overline {p}(\theta)\}$, where
$p(\theta)=\psi_r e^{ i \nu_r\tau_n\theta}$ is the eigenfunction
of $\mathcal{A}_{\tau_n}$ with respect to $i\nu_r\tau_n$, and the formal adjoint subspace of $P$ with respect to the
bilinear form \eqref{bil} is $P^*=\text{span}\{q(s),\overline{
q}(s)\}$, where $q(s)=\tilde\psi_r e^{ i \nu_r\tau_ns}$ is the
eigenfunction of $\mathcal{A}^*_{\tau_n}$ with respect to
$-i\nu_r\tau_n$. Denote $\Phi_I=(p(\theta),\overline p(\theta))$,
$\Psi_I=\ds\f{1}{\overline {S_n}(r)}(q(s),\overline {q}(s))^{T}$,
where $S_n(r)$ is defined as in Eq. \eqref{sn}, and then
$\langle\langle\Psi_I,\Phi_I\rangle\rangle= I$, where $I$ is the
identity matrix in $\mathbb{R}^{2\times2}$.

Note that formulas for the direction and
stability of Hopf bifurcation are all relative to $\gamma=0$ only,  let $\gamma=0$ in
Eq. \eqref{ab}, and we obtain a center manifold as follows
\begin{equation}\label{center}w(z,\overline
z)=w_{20}(\theta)\ds\frac{z^{2}}{2}+w_{11}(\theta)z\overline
z+w_{02}(\theta)\ds\frac{\overline z^{2}}{2}+O(|z|^3).\end{equation}
The solution semi-flow of Eq. \eqref{ab} on the center
manifold is
\begin{equation*}U_t=\Phi_I\cdot (z(t),\overline z(t))^{T}+w(z(t),\overline z(t)),\end{equation*}
where $z(t)$ satisfies
\begin{equation}\label{z(t)}
\begin{split}
\dot{z}(t) =&\ds\f{d}{dt}\langle\langle q(s),U_t\rangle\rangle\\
%=&\langle\langle q(s),
%\mathcal{A}_{\tau_n}U_t\rangle\rangle+\ds\f{1}{S_n(r)}\langle\langle q(s), X_0J(U_t,0)\rangle\rangle\\
%=&\langle\langle \mathcal{A}^*_{\tau_n}q(s),
%U_t\rangle\rangle+\ds\f{1}{S_n(r)}\langle q(0), J(U_t,0)\rangle_1\\
 =& i\nu_r\tau_n z(t)+\ds\f{1}{S_n(r)}\left\langle q(0),
J\left(\Phi_I(z(t),\overline z(t))^{T}+w(z(t),\overline
z(t)),0\right)\right\rangle_1.
\end{split}
\end{equation}
Denote
\begin{equation}
\begin{split}
g(z,\overline z)=&\ds\f{1}{S_n(r)}\left\langle q(0),
J\left(\Phi_I(z(t),\overline z(t))^{T}+w(z(t),\overline
z(t)),0\right)\right\rangle_1\\=&\sum_{2\le i+j\le 3}\f{g_{ij}}{i!j!}z^i\overline z^j+O(|z|^4).
\end{split}
\end{equation}
As in \cite{Chen2012}, we derive
\begin{equation}\label{gij}
\begin{split}
g_{20}=&-\ds\f{2r\tau_n}{S_n(r)}e^{-i\nu_r\tau_n}\int_0^L\int_0^L\overline{\tilde\psi_r}(x)\psi_r(x)K(x,y)e^{\alpha x+\alpha y}\psi_r(y)dxdy,\\
g_{11}=&-\ds\f{r\tau_n}{S_n(r)}e^{i\nu_r\tau_n}\int_0^L\int_0^L\overline{\tilde\psi_r}(x)\psi_r(x)K(x,y)e^{\alpha x+\alpha y}\overline{\psi_r}(y)dxdy\\
-&\ds\f{r\tau_n}{S_n(r)}e^{-i\nu_r\tau_n}\int_0^L\int_0^L\overline{\tilde\psi_r}(x)\overline{\psi_r}(x)K(x,y)e^{\alpha x+\alpha y}\psi_r(y)dxdy,\\
g_{02}=&-\ds\f{2r\tau_n}{S_n(r)}e^{i\nu_r\tau_n}\int_0^L\int_0^L\overline{\tilde\psi_r}(x)\overline{\psi_r}(x)K(x,y)e^{\alpha x+\alpha y}\overline{\psi_r}(y)dxdy,\\
g_{21}=&-\ds\f{2r\tau_n}{S_n(r)}\int_0^L\int_0^L\overline{\tilde\psi_r}(x)\psi_r(x)K(x,y)e^{\alpha x+\alpha y}w_{11}(-1)(y)dxdy\\
-&\ds\f{r\tau_n}{S_n(r)}\int_0^L\int_0^L\overline{\tilde\psi_r}(x)\overline{\psi_r}(x)K(x,y)e^{\alpha x+\alpha y}w_{20}(-1)(y)dxdy\\
-&\ds\f{r\tau_n}{S_n(r)}
e^{i\nu_r\tau_n}\int_0^L\int_0^L\overline{\tilde\psi_r}(x)w_{20}(0)(x)K(x,y)e^{\alpha x+\alpha y}\overline{\psi_r}(y)dxdy\\
-&\ds\f{2r\tau_n}{S_n(r)}e^{-i\nu_r\tau_n}\int_0^L\int_0^L\overline{\tilde\psi_r}(x)w_{11}(0)(x)K(x,y)e^{\alpha x+\alpha y}\psi_r(y)dxdy,
\end{split}
\end{equation}
where $w_{20}(\theta)$
and $w_{11}(\theta)$ are needed to be computed.

Note that $w(z(t),\overline{z}(t))$ satisfies
\begin{equation}\label{w}
\begin{split}
\dot w=&\mathcal{A}_{\tau_n}w+X_0J(\Phi_I(z,\overline z)^{T}+w(z,\overline z),0)\\
-&\Phi_I\langle\langle\Psi_I,X_0J(\Phi_I(z,\overline z)^{T}+w(z,\overline z),0)\rangle\rangle\\
=&\mathcal{A}_{\tau_n}w+H_{20}\ds\frac{z^2}{2}+H_{11}z\overline
z+H_{02}\ds\frac{\overline z^2}{2}+O(|z|^3),
\end{split}
\end{equation}
where $H_{20}$, $H_{11}$ and $H_{02}$ satisfy
\begin{equation*}
\begin{split}
&X_0J(\Phi_I(z,\overline z)^{T}+w(z,\overline z),0)
-\Phi\langle\langle\Psi,X_0J(\Phi_I(z,\overline z)^{T}+w(z,\overline
z),0)\rangle\rangle\\
=&H_{20}\ds\frac{z^2}{2}+H_{11}z\overline
z+H_{02}\ds\frac{\overline z^2}{2}+O(|z|^3).
\end{split}\end{equation*}
By using the chain rule, we see that $w$ also satisfies
\begin{equation*}\dot w=\ds\frac{\partial w(z,\overline z)}{\partial z}\dot z+\ds\frac{\partial w(z,\overline z)}{\partial \overline z}\dot{\overline
z}.
\end{equation*}
Therefore,
\begin{equation}\label{ws}
\begin{cases}
(2i\nu_r\tau_n-\mathcal{A}_{\tau_n})w_{20}=H_{20},\\
-\mathcal{A}_{\tau_n}w_{11}=H_{11}.\\
\end{cases}
\end{equation}
Note that for $\theta\in[-1,0)$,
\begin{equation}\label{H11}
\begin{split}
&H_{20}(\theta)=-(g_{20}p(\theta)+\overline g_{02}\overline
p(\theta)),\\
&H_{11}(\theta)=-(g_{11}p(\theta)+\overline g_{11}\overline
p(\theta)).
\end{split}
\end{equation}
Then, from Eq. \eqref{ws} and \eqref{H11}, $w_{20}$ and $w_{11}$ can be expressed as
\begin{equation}\label{W20}w_{20}(\theta)=\ds\frac{ig_{20}}{\nu_r\tau_n}p(\theta)+
\ds\frac{i\overline g_{02}}{3\nu_r\tau_n}\overline
p(\theta)+E_re^{2i\nu_r\tau_n\theta},
\end{equation}
and
\begin{equation}\label{W11}
w_{11}(\theta)=-\ds\frac{ig_{11}}{\nu_r\tau_n}p(\theta)+\ds\frac{i\overline
g_{11}}{\nu_r\tau_n}\overline p(\theta)+F_r.
\end{equation}
Noticing that
\begin{equation*}
H_{20}(0) =-\left(g_{20}p(0)+\overline g_{02}\overline
p(0)\right)-2r\tau_ne^{-i\nu_r\tau_n}\psi_r\int_0^L
K(x,y)e^{\alpha y}\psi_r(y)dy,
\end{equation*} we see from From Eqs. \eqref{w} and \eqref{ws} with $\theta=0$ that $E_r$ satisfies
\begin{equation*}(2i\nu_r\tau_n-\mathcal{A}_{\tau_n})
E_re^{2i\nu_r\tau_n\theta}\bigg\vert_{\theta=0}=-2r\tau_ne^{-i\nu_r\tau_n}\psi_r\int_0^L
K(x,y)e^{\alpha y}\psi_r(y)dy,
\end{equation*}
that is,
\begin{equation}\label{E}
\Delta(r,2i\nu_r,\tau_n)E_r=2r
e^{-i\nu_r\tau_n}\psi_r\int_0^L K(x,y)e^{\alpha y}\psi_r(y)dy.
\end{equation}
From Corollary \ref{c25}, we have that $2i\nu_r$ is not the
eigenvalue of $A_{\tau_n}(r)$, and hence
\begin{equation*}
E_r=2r
e^{-i\nu_r\tau_n}\Delta(r,2i\nu_r,\tau_n)^{-1}\left(\psi_r\int_0^L
K(x,y)e^{\alpha y}\psi_r(y)dy\right).
\end{equation*}
Similarly,
\begin{equation}\label{F1}
\begin{split}
F_r=&r\Delta(r,0,\tau_n)^{-1}\left(
e^{i\nu_r\tau_n}\psi_r\int_0^L
K(x,y)e^{\alpha y}\overline{\psi_r}(y)dy\right)\\
+&r\Delta(r,0,\tau_n)^{-1}\left(e^{-i\nu_r\tau_n}\overline{\psi_r}\int_0^L
K(x,y)e^{\alpha y}\psi_r(y)dy\right).
\end{split}
\end{equation}
Then, $E_r$ and $F_r$ can be derived in the following.
\begin{lemma}\label{comf}
For $r\in(0,r_2],$  let $E_r$ and $F_r$ be defined as in \eqref{E} and \eqref{F1}.
Then
\begin{equation}\label{EFgostar}
E_r=b_rc_0+\phi_r,\;\;
\end{equation}
where $c_0$ is defined as in Eq. \eqref{ovec}, $\phi_r\in (X_1)_{\mathbb C}$, and $b_r$, $\phi_r$ satisfy
$$\lim_{r\to0} b_r=\ds\f{2i}{1-2i},\;\;\lim_{r\to0} \|\phi_r\|_{Y_{\mathbb C}}=0,$$
and $\lim_{r\to 0} \|F_r\|_{Y_{\mathbb C}}=0$.
\end{lemma}
\begin{proof} We only prove the estimate for $E_r$, and $F_r$ can be derived similarly.
Substituting Eq. \eqref{EFgostar} in to Eq. \eqref{E}, we have
\begin{equation}\label{cephi}
\begin{split}&\ds\f{1}{r}P_0\phi_r=-e^{\alpha x}\Tilde K(r)(b_rc_0+\phi_r)+
u_re^{\alpha x}\int_{0}^LK(x,y)e^{\alpha y}[b_r
c_0+\phi_{r}(y)]dye^{-2i\nu_r\tau_n}\\
&~~~~~~+2ih_re^{\alpha x}(b_rc_0+\phi_r)+2
e^{-i\nu_r\tau_n}\psi_re^{\alpha x}\int_0^L
K(x,y)e^{\alpha y}\psi_r(y)dy,\end{split}
\end{equation}
where $h_r$ is defined as in Theorem
\ref{cha}.
Integrating Eq. \eqref{cephi} over
$(0,L)$, and noticing that $|h_r|$, $\|u_r\|_\infty$ and $\|\psi_r\|_\infty$ are bounded for $r\in(0,r_2]$, we see that there exist constants
$M_0,~M_1>0$
such that
\begin{equation}\label{bet}
|b_r|\le
M_0\|\phi_{r}\|_{Y_\mathbb{C}}+M_1,
\end{equation}
for any $r\in(0,r_2]$.
Multiplying Eq. \eqref{cephi} by $\overline \phi_r$, and integrating the result over $(0,L)$, we see from Lemma \ref{lem21} and Eq. \eqref{bet} that there exist
constants
$M_2,~M_3>0$
such that  $$\la_2\|\phi_r\|^2_{Y_{\mathbb C}}\le
rM_2\|\phi_{r}\|^2_{Y_\mathbb{C}}+rM_3\|\phi_{r}\|_{Y_\mathbb{C}},$$
for any $r\in(0,r_2]$, where $\la_2$ is defined as in Lemma \ref{lem21}.
This leads to $\lim_{r\to0} \|\phi_r\|_{Y_{\mathbb C}}=0$. Then, integrating Eq. \eqref{cephi} over $(0,L)$, and taking the limit
of the equation at both side as $r\to0$, we obtain
\begin{equation*}
(1-2i)\left(\lim_{r\to0}b_r\right)\int_0^Le^{\alpha x}dx=2i\int_0^Le^{\alpha x}dx,
\end{equation*}
which leads to $\lim_{r\to0} b_r=\f{2i}{1-2i}$. Similarly, we can prove that $\lim_{r\to 0} \|F_r\|_{Y_{\mathbb C}}=0$.
\end{proof}

Therefore, by similar arguments similar to \cite{Chen2012}, one can also derive
\begin{equation}\label{glas}
\lim_{r\to 0}g_{11}=0\;\;\text{and}\;\;\lim_{r\to
0}\mathcal{R}e[g_{21}]<0.
\end{equation}
It follows from \cite{hassard1981theory,wu1996theory} that $C_1(0)$ determines the direction and
stability of bifurcating periodic orbits, where
\begin{equation*}
C_1(0)=\dfrac{i}{2\nu_r\tau_n}\left(g_{11}g_{20}-2|g_{11}|^2
-\dfrac{|g_{02}|^2}{3}\right)+\dfrac{g_{21}}{2}.
\end{equation*}
Then, Eq. \eqref{glas} implies $\lim_{r\to0}\mathcal{R}e[C_1(0)]<0$. Hence we have the
following result.
\begin{theorem}\label{T3.3}
Assume that $r\in(0,r_2]$, where $0<r_2\ll 1$. Let
$\{\tau_n(r)\}_{n=0}^\infty$ be the Hopf bifurcation points of Eq. \eqref{delay} obtained in Theorem \ref{thm37}.
Then, for each $n\in \mathbb N \cup \{0\}$, the direction of the
Hopf bifurcation at $\tau=\tau_n$ is forward and the bifurcating periodic solutions
from $\tau=\tau_0$ is orbitally asymptotically stable.
\end{theorem}
\section{The effect of spatial heterogeneity}
In this section, we will consider the effect of spatial heterogeneity on Hopf bifurcation values.
It follows from Lemma \ref{l25}, Theorems \ref{thm36} and \ref{thm37} that the first Hopf bifurcation value $\tau_0$ of Eq. \eqref{delay} depends on $r$, $\alpha$, $L$, and satisfies:
\begin{equation}\label{ho0ii}
\begin{split}
\tau_0(r,\alpha,L)=&\frac{\theta_r(\alpha,L)}{r h_r(\alpha,L)},\;\;\lim_{r\to0}\theta_r(\alpha,L)=\ds\f{\pi}{2},\\
\lim_{r\to0}h_r(\alpha,L)=&h_{0}(\alpha,L)=\ds\f{\int_0^L m(x)e^{\alpha x}dx}{\int_0^Le^{\alpha x}dx}.
\end{split}
\end{equation}
If $m(x)\equiv m_0$, where $m_0$ is a positive constant, then
$$h_0(\alpha,L)=m_0\; \;\text{and}\;\;\lim_{r\to0} r\tau_0(r,\alpha,L)=\ds\f{\pi}{2m_0}$$
for any $\alpha\in(-\infty,\infty)$ and $L>0$, and hence $\tau_0(r,\alpha,L)\approx \ds\f{\pi}{2rm_0 }$ for small $r$.
It seems that the value of $\tau_0(r,\alpha,L)$ has no significant change
as advection $\alpha$ or spatial scale $L$ changes, when $m(x)$ is spatially homogeneous.

Then we consider the case that $m(x)$ is spatially heterogeneous. We find that
Hopf bifurcation is more likely to occur as spatial scale $L$ increases, if $m(x)$ achieve its maximum
at boundary $x=L$.
\begin{proposition}
Suppose that $m(x)$ is non-constant, $m(L)=\max_{x\in[0,L]} m(x)$, $\alpha\in(-\infty, \infty)$, and $L_1>L_2>0$. Then there exists $\tilde r>0$, depending on $L_1$, $L_2$ and $\alpha$, such that $\tau_0(r,\alpha,L_1)<\tau_0(r,\alpha, L_2)$ for $t\in(0,\tilde r]$.
\end{proposition}
\begin{proof}
Since \begin{equation*}
\ds \f{\partial h_0(\alpha,L)}{\partial L}=\ds\f{e^{\alpha L}\int_0^L\left[m(L)-m(x)\right]e^{\alpha x}dx}{\left(\int_0^L e^{\alpha x} dx\right)^2}>0,
\end{equation*}
we see that, for any fixed $\alpha\in(-\infty,\infty)$, $h_0(\alpha, L)$ is strictly increasing for $L\in(0,\infty)$. Note that
$$\tau_0(r,\alpha,L)=\frac{\theta_r(\alpha,L)}{r h_r(\alpha,L)}\;\;\text{and}\;\;\lim_{r\to0}r\tau_0(r,\alpha,L)=\ds\f{\pi}{2h_0(\alpha,L)}.$$
It follows that there exists $\tilde r>0$, depending on $L_1$, $L_2$ and $\alpha$, such that $\tau_0(r,\alpha,L_1)<\tau_0(r,\alpha, L_2)$ for $t\in(0,\tilde r]$.
\end{proof}
In the following we will choose different types of $m(x)$ to show the effect of spatial heterogeneity.

\begin{example}\label{exam3} Choose
\begin{equation}\label{p13ii}
m(x)=x.
\end{equation}
In this case,
\begin{equation*}
\begin{split}
h(\alpha,L)=&\ds\f{\alpha L e^{\alpha L}-e^{\alpha L}+1}{\alpha(e^{\alpha L}-1)},\;\;h_0(0,L)=\ds\f{L}{2},\\
\ds \f{\partial h_0(\alpha,L)}{\partial a}=&\ds\frac{\int_0^L x^2e^{\alpha x} dx\int_0^L e^{\alpha x}dx-\left(\int_{0}^L x e^{\alpha x}dx\right)^2}{\left(\int_0^L e^{\alpha x}dx\right)^2}>0,\\
\ds \f{\partial h_0(\alpha,L)}{\partial L}=&\ds\f{e^{\alpha L}\left(e^{\alpha L}-\alpha L -1\right)}{\left (e^{\alpha L}-1\right)^2}>0.
\end{split}
\end{equation*}
Consequently, if we choose \begin{equation}\label{p14ii}
m(x)=m_0-x,
\end{equation}
where $m_0$ is a constant and $m_0>L$,
then
\begin{equation*}
\ds \f{\partial h_0(\alpha,L)}{\partial a}<0,\;\;\ds \f{\partial h_0(\alpha,L)}{\partial L}<0.
\end{equation*}
Then we have the following two statements on the effect of advection $\alpha$.
\begin{enumerate}
\item Assume that $L\in(0,\infty)$, $m(x)=x$ and $\alpha_1>\alpha_2$. Then there exists $\tilde r>0$, depending on $\alpha_1$, $\alpha_2$ and $L$, such that $\tau_0(r,\alpha_1,L)<\tau_0(r,\alpha_2,L)$ for $r\in(0,\tilde r]$.
\item Assume that $L\in(0,\infty)$, $m(x)=m_0-x$, where $m_0>L$, and $\alpha_1>\alpha_2$. Then there exists $\tilde r>0$, depending on $\alpha_1$, $\alpha_2$ and $L$, such that $\tau_0(r,\alpha_1,L)>\tau_0(r,\alpha_2,L)$ for $r\in(0,\tilde r]$.
\end{enumerate}
Therefore, Hopf bifurcation is more likely to occur when the advection rate increases (respectively, decreases) for $m(x)=x$ (respectively, $m(x)=m_0-x$, where $m_0>L$).
Similarly, we have the following two statements on the effect of spatial scale $L$.
\begin{enumerate}
\item Assume that $\alpha\in(-\infty,\infty)$, $m(x)=x$ and $L_1>L_2$. Then there exists $\tilde r>0$, depending on $L_1$, $L_2$ and $\alpha$, such that $\tau_0(r,\alpha,L_1)<\tau_0(r,\alpha,L_2)$ for $r\in(0,\tilde r]$.
\item Assume that $\alpha\in(-\infty,\infty)$, $m(x)=m_0-x$, where $m_0>L$, and $L_1>L_2$. Then there exists $\tilde r>0$, depending on $L_1$, $L_2$ and $\alpha$, such that $\tau_0(r,\alpha,L_1)>\tau_0(r,\alpha,L_2)$ for $r\in(0,\tilde r]$.
\end{enumerate}
Therefore, Hopf bifurcation is more likely to occur when spatial scale $L$ increases (respectively, decreases) for $m(x)=x$ (respectively, $m(x)=m_0-x$, where $m_0>L$).

\end{example}

\begin{example}\label{exam2} Choose
\begin{equation}\label{p13i}
m(x)=\sin \f{\pi x}{L}.
\end{equation}
In this case,
\begin{equation*}
h(\alpha,L)=\ds\f{\pi\alpha L \left(e^{\alpha L}+1\right)}{\left(\pi^2+\alpha^2L^2\right)(e^{\alpha L}-1)},\;\;h_0(0,L)=\ds\f{2}{\pi}.
\end{equation*}
Therefore, if $\alpha L>\pi$, then
$$\ds \f{\partial h_0(\alpha,L)}{\partial a}<0\;\;\text{and}\;\;\ds \f{\partial h_0(\alpha,L)}{\partial L}<0.$$
Consequetly, we have the following two statements on the effects of advection $\alpha$ and spatial scale $L$.
\begin{enumerate}
\item Assume that $\alpha_1>\alpha_2>\pi/L$. Then there exists $\tilde r>0$, depending on $\alpha_1$, $\alpha_2$ and $L$, such that $\tau_0(r,\alpha_1,L)>\tau_0(r,\alpha_2,L)$ for $r\in(0,\tilde r]$.
\item Assume that $L_1>L_2>\pi/\alpha$. Then there exists $\tilde r>0$, depending on $L_1$, $L_2$ and $\alpha$, such that $\tau_0(r,\alpha_1,L)>\tau_0(r,\alpha_2,L)$ for $r\in(0,\tilde r]$.
\end{enumerate}
Therefore, Hopf bifurcation is more likely to occur when advection rate $\alpha>\pi/L$ decreases or spatial scale $L>\pi/\alpha$ decreases.

\end{example}

\begin{center}
\vskip1cm {\small
\bibliographystyle{amsalpha}
}
\end{center}

\end{document}